\documentclass[reqno]{amsart}
\usepackage{amssymb,amsmath,amsthm,amstext,amsfonts}
\usepackage[dvips]{graphicx}
\usepackage[ansinew]{inputenc}
\usepackage{color}

\makeatletter \@addtoreset{equation}{section} \makeatother

\renewcommand\thetable{\thesection.\@arabic\c@table}

\theoremstyle{plain}
\newtheorem{maintheorem}{Theorem}

\newtheorem{theorem}{Theorem }[section]
\newtheorem{proposition}[theorem]{Proposition}
\newtheorem{lemma}[theorem]{Lemma}
\newtheorem{corollary}[theorem]{Corollary}
\newtheorem{claim}{Claim}
\theoremstyle{definition} \theoremstyle{remark}

\newcommand{\eps}{\varepsilon}

\newcommand{\Z}{\mathbb{Z}}
\newcommand{\N}{\mathbb{N}}
\newcommand{\R}{\mathbb{R}}

\newcommand{\ov}{\overline}

\newcommand{\fix}{\operatorname{Fix}}
\newcommand{\dif}{\operatorname{Diff}^1(M)}
\newcommand{\sink}{\operatorname{Sink}}
\newcommand{\tcap}{\pitchfork}

\newcommand{\cR}{\mathcal{R}}
\newcommand{\cP}{\mathcal{P}}

\newcommand{\cM}{\mathcal{M}}
\newcommand{\cB}{\mathcal{B}}

\newcommand{\cU}{\mathcal{U}}

\newcommand{\cA}{\mathcal{A}}

\title{Dirac physical measures for generic diffeomorphisms}

\author{Bruno Santiago}
\address{Universit\'e de Bourgogne}
\email{bruno.rodrigues-santiago@u-bourgogne.fr}
\date{\today}

\begin{document}

\maketitle

\begin{abstract}
We prove that, for a $C^1$ generic diffeomorphism, the only Dirac physical 
measures with dense statistical basin are those supported on sinks.
\end{abstract} 

{\footnotesize
}

\section{Introduction}

One of the paradigms of dynamical systems in the last decades has been the Palis program \cite{palis},
which aims to describe the behaviour of every element in a large subset (open and dense, residual etc...) of the space dynamical systems,
mainly diffeomorphisms and vector fields endowed with the $C^r$ topology, $r\geq 0$.

Within the approaches to understand a dynamical system the statistical description of its orbits, that is the accumulation set of the sequence of measures
$$\frac{1}{n}\sum_{l=1}^{n-1}\delta_{f^l(x)},$$
where $\delta_y$ stands for the the Dirac mass concentrated in $y$,
plays a major role.
For instance, it is a conjecture of Palis \cite{palis} that densely in the space of $C^r$ diffeomorphisms of a closed manifold $M$, every element
displays a finite number invariant probability measures which capture the statistics of Lebesgue almost every orbit. 

Moreover, this statistical description, through the works of Sinai, Ruelle and Bowen, was one the most important achievements in the theory of Axiom A diffeomorphisms. 
They established, for every $C^2$ hyperbolic attractor, the existence of a probability measure fully supported in the attractor which disintegrates in an absolutely continuous manner 
with respect to the volume along unstable manifolds. This property, together with the hyperbolicity, implies that the measure captures the statistics of almost every orbit in the basin of attraction. This type of measure is nowadays called SRB.
See for instance \cite{Bowen} and \cite{Young}. 

In contrast to the SRB measures, which are characterized by a strong and rich geometrical property, that requires the existence of unstable manifolds, we have the Dirac mass concentrated at attractive fixed (or periodic) points. 
The paradigmatic example being the north-south dynamics in the sphere, in which the sink captures all the future dynamics, from both the topological and the statistical points of views. Moreover, this dynamics is structurally stable and thus
it persists by small perturbations.

However, there are several examples of non-attractive fixed points which captures almost all statistical behaviour. In the paper \cite{hy} Hu and Young construct a
diffeomorphism of $\mathbb{T}^2$ having a physical measure, with full statistical basin, supported in a fixed point which is
not a sink (it has one contracting direction and a parabolic ``almost unstable'' direction). The example is even topologically conjugated to a linear
Anosov map. A more simple construction presenting the same phenomenon is the figure-eight attractor of
Bowen (see the introduction of \cite{k}), where the fixed point is a saddle whose stable and unstable manifolds form a double connection.

Another construction was given by Saghin, Sun and Vargas \cite{ssv} and Saghin and Vargas \cite{sv}. They built several types of transitive flows with 
Dirac physical measures, which in particular are not attractive fixed points.
It is worth to remark that all their examples present some type of saddle connection. An even more surprising example comes from interval dynamics. 
Hofbauer and Keller showed the existence of quadratic maps with a Dirac physical measure at an unstable fixed point. 
Again, it is worth to remark that within quadratic maps this phenomenon has zero Lebesgue measure. We mention also the work Colli and Vargas \cite{CV}. Using a Newhouse toy model they build a 
hyperbolic saddle type fixed point which is a statistical attractor whose basin has non-empty interior. 
Although in this case the basin is not full, it captures the statistics of every point in some wondaring domain.

All these examples share an interesting feature: they may be complicated from the topological point of view, but from the statistical point of view they are completely trivial. 
The asymptotic statistical behaviour of Lebesgue almost every orbit is captured by a fixed point, which in fact is the sole statistical attractor of the dynamics. 

Nonetheless, each one of them present some clearly non-generic mechanism. It is then natural to ask: is it possible to find a locally 
residual set of diffeomorphisms presenting a statistical attractor which is a hyperbolic saddle type fixed point?

In this note, we give a partial negative answer which includes the case where the basin is full.

\begin{maintheorem}
\label{t.main}
Let $f$ be a $C^1$ generic diffeomorphism. Assume that there exists 
$\sigma\in\fix(f)$ such that $\delta_{\sigma}$ is a physical measure with dense statistical 
basin. Then $\sigma$ is a sink.
\end{maintheorem}

As an immediate corollary, we obtain that, apart from sinks, there is no locally generic example of a fixed point which is a statistical attractor with full basin.

\begin{corollary}
 \label{c.cor}
Let $f$ be a $C^1$ generic diffeomorphism. Assume that there exists 
$\sigma\in\fix(f)$ such that $\delta_{\sigma}$ is a physical measure with full statistical basin. Then $\sigma$ is a sink. 
\end{corollary}

Let us briefly comment on why we obtain only a partial answer. Our method of proof involves the existence a Lyapunov stable set containing the saddle, inside which we can perform an entropy estimation.
For this we use the theorems \cite{MP} and \cite{BC}, 
which require obtaining a residual set of points in the manifold which accumulates in the saddle. 
For obtaining this set, we use our denseness assumption. It goes without saying that it would very nice to obtain a full answer. However, this would require strong improvements of our techniques.

This note is structured in the following way. In Section~\ref{notations} we fix notations and give precise definitions of the objects apearing in this work. In Section~\ref{nomeruim}
we derive some consequences of the existence of a statistical attractor which is a fixed points (this uses an interesting, although elementary, result contained in \cite{ara}). 
In Section~\ref{generic} we collect some $C^1$ generic theorems and in Section~\ref{prova} we give the proof of Theorem~\ref{t.main}. 

\section*{Acknowledgments} I thank Pablo Guarino for many useful discussions and for his comments on a preliminary version of this paper. 

\section{Notations and definitions}\label{notations}

Let $M$ be a closed manifold of dimension $d$. We denote by $\dif$ the space of diffeomorphisms over $M$, endowed with the $C^1$ topology.
Given $f\in\dif$ and $x\in M$, the orbit of $x$ is the set $O(x)=\{f^n(x);n\in\Z\}$.

We denote by $\fix(f)$ the set of fixed points of $f$. Recall that a periodic point is an element $p\in\fix(f^n)$, for some integer $n>0$. The smalest of such $n$
is called the period of $p$, and is denoted by $\pi(p)$.

We denote by $m$ the normalised Lebesgue measure of $M$.

\subsection{The weak star topology}

$\cP(M)$ denotes the set of propability measures, endowed with the weak-star topology. Recall that, by compactness of $M$, we can give a metric generating this topology in
the following manner. Fix $\{\varphi_n\}_{n\in\N}$ a countable dense subset of $C^0_1(M)$, the space of continuous functions over $M$ bounded 
by $1$. Then, given $\mu,\nu\in\cP(M)$, putting 
$$d(\mu,\nu):=\sum_{n=1}^{\infty}2^{-n}\left|\int\varphi_nd\mu-\int\varphi_nd\nu\right|$$ 
defines a metric which gives the weak star topology. We denote by $\cP_f(M)$ the subset of $\cP(M)$ formed by measures which are invarinat under the element $f\in\dif$.

\subsection{Assimptotic measures}

In this paper any convergence of probability measures will be with respect to the weak star topology.

Given $x\in M$, consider the probability measures $\mu_k(x):=\frac{1}{n}\sum_{l=1}^{k-1}\delta_{f^l(x)}$, where $\delta_y$ is the Dirac mass concentrated in the point $y$.  
An \emph{assimptotic measure} of $x$ is an accumulation point of the sequence $\mu_k(x)$.
We denote by $\cM(x)$ the set of assimptotic measures of $x$. 

Consider now an invariant probability measure $\mu\in\cP_f(M)$. The \emph{statistical basin} of $\mu$ is the set 
$$\cB(\mu)=\left\{x\in M;\cM(x)=\{\mu\}\right\}.$$
We say that $\mu$ is a \emph{physical measure} if $m(\cB(\mu))>0$.

\subsection{Lyapunov exponents}

For $x\in M$ and $v\in T_xM\setminus\{0\}$, the \emph{Lyapunov exponent} of at $x$ in the direction of $v$ is 
$$\lambda(x,v):=\lim_{n\to\infty}\frac{1}{n}\log\|Df^n(x)v\|,$$
whenever the limit exists. By Oseledets' theorem given $\mu\in\cP_f(M)$ there exists a full measure set, called the set of \emph{regular points}, and measurable functions
$\chi_1\leq...\leq\chi_d:M\to\R$, such that given a regular point $x\in M$, for every $v\in T_xM\setminus\{0\}$ there exists $i$ such that 
$$\chi_i(x)=\lambda(x,v).$$
In the particular case $\mu=\delta_{\sigma}$, for some $\sigma\in\fix(f)$, the Lyapunov exponents are the logarithm of the modulus of the eigenvalues of $Df(\sigma)$. For details, see \cite{M}

\subsection{Homoclinic classes and dominated splittings}

Given a compact $\Lambda\subset M$, invariant under $f\in\dif$, we say that $\Lambda$ admits a \emph{dominated spliting} if there exists a decomposition of the tangent bundle
$T_{\Lambda}M=E\oplus F$, which is invariant under the derivative $Df$, and numbers $C>0$, $0<\lambda<1$ such that for every $x\in\Lambda$ and $n>0$ one has 
$$\|Df^n(x)|_E\|\|Df^{-n}(f^n(x))|_F\|\leq C\lambda^n.$$

In particular, given a hyperbolic periodic point $p$ (i.e no eigenvalue of $Df^{\pi(p)}(p)$ have moddulus $1$) its orbit admits a dominated splitting $E^s\oplus E^u$. 
The stable manifold theorem \cite{pdm} gives submanifolds $W^s(p)$ and $W^u(p)$, which are tangent to $E^s$ and $E^u$, respectively, at $p$. We denote 
$W^a(O(p))=\cup_{l=0}^{\pi(p)-1}W^a(f^l(p))$, $a=s,u$.

Given two hyperbolic periodic points $p$ and $q$ we say that they are \emph{homoclinically related} if $W^s(O(p))\tcap W^u(O(q))\neq\emptyset$ and
$W^u(O(p))\tcap W^s(O(q))\neq\emptyset$. The \emph{homoclinic class} of a periodic point $p$, denoted by $H(p)$, is the closure of the set of periodic points $q$, homoclinically
related with $p$.

We say that a hyperbolic saddle type periodic $p$ is \emph{dissipative} if $|\det{Df^{\pi(p)}}(p)|<1$

\subsection{Lyapunov stable sets}

A compact $\Lambda\subset M$, invariant under $f\in\dif$, is said to be \emph{Lyapunov stable} if for every neighborhood $U$ of $\Lambda$ it is possible to
find a neighborhood $V$ of $\Lambda$ such that if $x\in V\cap U$ then $f^n(x)\in U$, for every $n>1$. 

\section{Dirac physical measures}\label{nomeruim}

In this section we give two consequences of the existence of a Dirac physical measure.

\begin{proposition}
\label{p.alrozero}
For every $f\in\dif$ and every $\sigma\in\fix(f)$, if $\delta_{\sigma}$ is a physical measure then $|\det(Df(\sigma))|\leq 1$.
\end{proposition}

The proof of this proposition needs a result contaned in Ara\'ujo's thesis \cite{ara}. We state and prove this result here for the sake of completeness. 
Take a real number $\delta>0$ and consider the set $\Lambda(\delta,f)$ of points $x\in M$ for which there exists a positive integer $N$ such that
$n\geq N$ implies $|\det{Df^n(x)}|<(1+\delta)^n$. The result of Ara\'ujo's thesis is the following (see also Lemma 2 in \cite{amseu} for a flow version):

\begin{lemma}
 \label{l.ara}
For every $\delta>0$ one has 
\begin{enumerate}
 \item $m\left(\Lambda(\delta,f)\right)=1$. 
 \item if $x\in\Lambda(\delta,f)$ and $\mu\in\cM(x)$ then $\int\log|\det{Df}|d\mu\leq\delta$
\end{enumerate}
\end{lemma}
\begin{proof}
For each positive integer $n$ consider the set 
$$\Gamma_n=\left\{x\in M;|\det{Df^n(x)}|\geq(1+\delta)^n\right\}.$$
Since $f^n$ is a diffeomorphism, one can apply the change of variable formula to the unity constant function and obtain
$$1=\int_Mdm=\int_M|\det{Df^n}|dm\geq\int_{\Gamma_n}|\det{Df^n(x)}|dm\geq(1+\delta)^nm\left(\Gamma_n\right),$$
thus 
$$m\left(\Gamma_n\right)\leq\frac{1}{(1+\delta)^n}.$$
Since, by definition, $\Lambda(\delta,f)=\cup_{N\in\N}\cap_{n\geq N}(M\setminus\Gamma_n)$, Borel-Canteli's theorem implies that $m\left(\Lambda(\delta,f)\right)=1$. 

Now, pick a point $x\in\Lambda(\delta,f)$ and a measure $\mu\in\cM(x)$. This means that there exists a sequence of positive integers $n_i\to+\infty$ such that
$\mu=\lim_{i\to\infty}\frac{1}{n_i}\sum_{l=1}^{n_i-1}\delta_{f^l(x)}.$ As $x\in\Lambda(\delta,f)$, for $i$ sufficiently large we have that
$$\frac{1}{n_i}\sum_{l=1}^{n_i-1}\log{|\det{Df(f^l(x))}|}=\frac{1}{n_i}\log{|\det{Df^{n_i}(x)}|}<\log(1+\delta),$$
and therefore
$$\int\log|\det{Df}|d\mu\leq\log(1+\delta)\leq\delta.$$
\end{proof}

\begin{proof}[Proof of Proposition~\ref{p.alrozero}]
Consider $\Lambda(f):=\cap_{n\in\N}\Lambda(1/n,f)$. By Lemma~\ref{l.ara} one has that $m(\Lambda(f))=1$ and if $x\in\Lambda(f)$ and $\mu\in\cM(x)$ then 
$\int\log|\det{Df}|d\mu\leq 1/n$, for every $n\in\N$. 

Now, since $\cB(\delta_{\sigma})$ has positive measure, there exists a point $x\in\cB(\delta_{\sigma})\cap\Lambda(f)$. This implies that 
$$\log|\det{Df(\sigma)}|=\int\log|\det{Df}|d\delta_{\sigma}\leq 0,$$
which concludes the proof.
\end{proof}

Below, we use our denseness assumption to construct a residual set of points which accumulates in the fixed point. The key is the following general statement.

\begin{lemma}
\label{l.baciafraca}
Let $f:M\to M$ be a continuous map. Assume that there exists $\mu\in\cP_f(M)$ such that $\cB(\mu)$ is dense. 
Then, $\{x\in M;\mu\in\cM(x)\}$ is a residual subset of $M$. 
\end{lemma}
\begin{proof}
Notice that $\{x\in M;\mu\in\cM(x)\}$ can be writen as the countable intersection 
$$\bigcap_{n\in\N}\bigcap_{m\in\N}\bigcup_{k\geq m}B(n,k),$$
where $B(n,k)=\left\{x\in M;d(\mu_k(x),\mu)<\frac{1}{n}\right\}$. Moreover, one easily sees that each $B(n,k)$ is an open subset of $M$. Since, 
by assumption, the set $\{x\in M;\mu\in\cM(x)\}$ already contains a dense subset of $M$, the result follows.
\end{proof}

As an immediate consequence we get.

\begin{corollary}
 \label{c.acumula}
For every $f\in\dif$ and every $\sigma\in\fix(f)$, if $\cB(\delta_{\sigma})$ is dense then there exists a residual set $L_f\subset M$ such that 
if $x\in L_f$ then $\sigma\in\omega(x)$.
\end{corollary}

\section{$C^1$ generic results}\label{generic}

In this section we collect some $C^1$-generic results that we shall combine togehter to obtain our main theorem.

\subsection{Dominated splitting far sinks}

The lemma below is a consequence of arguments of R. Potrie \cite{potrie}, applied here in a simpler situation.

\begin{lemma}
\label{l.bobo}
There exists a residual subset $\cR_1\subset\dif$ such that if $f\in\cR_1$ has no sinks then
for every dissipative $\sigma\in\fix(f)$
its homoclinic class $H(\sigma)$ admits a dominated splitting $E\oplus F$ such that $\det(Df(\sigma)|_{F})>1$.
\end{lemma}

\subsubsection{Tools}

The lemma below says that one can change, by a small perturbation, one of the eigenvalues of the derivative at a periodic point while keeping the others fixed. 
It will be used several times in the sequel.
Since it is an easy application of Frank's lemma \cite{Franks}, we omit the proof.  

\begin{lemma}
 \label{l.frank}
Let $p$ be a periodic point of $f\in\dif$. Take $\cU$ a neighborhood of $f$ and $U$ a neighborhood of $O(p)$. Consider 
$$0<\lambda_1<...<\lambda_k$$
the set of modulus of eigenvalues of $Df^{\pi(p)}(p)$. Then, there exists $\eps>0$ such that for every $i=1,...,k$ and every
$\lambda^*\in(\lambda_i-\eps,\lambda_i+\eps)$ there exists $g\in\cU$, which coincides with $f$ over $(M\setminus U)\cup O(p)$
and such that the set of modulus of eigenvalues of $Dg^{\pi(p)}(p)$ is
$$\lambda_1<...<\lambda_{i-1}<\lambda^*<\lambda_{i+1}<...<\lambda_k.$$
\end{lemma}

The second tool we shall use enables us to obtain the dominated splitting when we are far from sinks. It  is Theorem 2 in \cite{bobo}.

\begin{theorem}[Bochi-Bonatti]
\label{t.bobo}
Let $f$ be a diffeomorphism of a $d$-dimensional compact manifold, and $\gamma_n=O(p_n)$ be a sequence of periodic orbits whose periods 
$\pi(p_n)$ tends to infinity and that converges in the Hausdorff topology to a compact set $\Lambda$. Let 
$$E_1\oplus...\oplus E_m$$
be the finest dominated splitting over $\Lambda$. Then, given $\eps>0$ there is $N$ such that for every $n\geq N$ there exists a 
$\eps$-$C^1$-perturbation $g$ of $f$ with support in an arbitralily small neighborhood of $\gamma_n$, preserving the orbit $\gamma_n$,
and such that, for every $i=1,...,m$, all the eigenvalues of $Dg^{\pi(p_n)}(p_n)$ in the subspace $E_i$ are equal to 
$\sqrt[\dim(E_i)]{\det(Df^{\pi(p_n)}(p_n)|_{E_i})}.$
\end{theorem}

\subsubsection{Proof of Lemma~\ref{l.bobo}}

By a classical Baire argument\footnote{See for instance \cite{semi}.}, since the map $f\mapsto\overline{\sink(f)}$ is lower semicontinuous, there exists
a resudual subset $\cR_1\subset\dif$ such that if $f\in\cR_1$ has no sinks then there exists a neighborhood 
$\cU$ of $f$ in $\dif$ such that no diffeomorphism $g\in\cU$ has sinks. 

We fix $f\in\cR_1$ and let $\sigma\in\fix(f)$ be a dissipative saddle. Then $\det(Df(\sigma))<1$. If $H(\sigma)=\{\sigma\}$ the existence 
of a dominated splitting is automatic, therefore we can assume that $H(\sigma)\neq\{\sigma\}$. 

Assume by contradiction that $H(\sigma)$ has no dominated splitting.
We shall apply Theorem \ref{t.bobo} to find a contradiction. 

\begin{claim}
There exists a sequence
$\gamma_n=O(p_n)$ of periodic orbits of $f$ having the following properties
\begin{itemize}
 \item each $p_n$ is homoclinically related with $\sigma$
 \item $\gamma_n$ converges in the Hausdorff topology to $H(\sigma)$
 \item $\pi(p_n)\to\infty$
 \item  $\det(Df^{\pi(p_n)}(p_n))<1$.
\end{itemize} 
\end{claim}
\begin{proof}
This result is contained in Lemma 1.10 of \cite{BDP}. We give a brief sketch of the proof for the comfort of the reader.  
By the shadowing lemma and by hyperbolicity we can create a periodic orbit $\eps$-dense in this horseshoe, and homoclinically related with $\sigma$.
Moreover, these periodic orbits can be chosen so that they shadow $\sigma$ during an arbitrarily large proportion of their orbit, implying that they are dissipative
Letting $\eps$ tend to zero, one creates the desired sequence. It remains to notice that, as $H(\sigma)\neq\{\sigma\}$, the periods of these orbits are unbounded.
\end{proof}

Since $H(\sigma)$ has no dominated splitting, a direct application of Theorem \ref{t.bobo} implies that 
there is $g\in\cU$ with a sink, a contradiction. Thus, there exists a non-trivial finest dominated splitting $E_1\oplus...\oplus E_m\oplus F$
over $H(\sigma)$. 

\begin{claim}
$\det(Df(\sigma)|_{F})>1$. 
\end{claim}
\begin{proof}
Assume once more by contradiction that $\det(Df(\sigma)|_{F})\leq 1$. As before, we get a sequence of periodic orbits
$\gamma_n=O(p_n)$, converging to $H(\sigma)$ and such that $\det(Df^{\pi(p_n)}(p_n)|_{F})<1+\eps$, with $\eps>0$ as small as we please.
Applying Theorem \ref{t.bobo} once more, we obtain $g$ close to $f$, such that $\gamma_n$, with $n$ large, is a periodic orbit for $g$,
and all the eigenvalues of $Dg^{\pi(p_n)}(p_n)$ in the subspace $F$ are bounded by $1+\eps$. Since $E_1\oplus...\oplus E_m\oplus F$
is a dominated splitting, all the remaining eigenvalues of $Dg^{\pi(p_n)}(p_n)$ are bounded by $1+\eps$. Thus, applying Lemma~\ref{l.frank}, by another
small perturbation we can create a sink, for some $\ov{g}\in\cU$, a contradiction.
\end{proof}

The proof is now complete.
\qed

\subsection{Existence of Lyapunov stable sets}

The result below is our source of Lyapunov stability. See \cite{MP} and Corollaire 1.8 in \cite{BC}.

\begin{theorem}
\label{t.mp}
There exists a residual set $\cR_2\subset\dif$ such that for every $f\in\cR_2$, for a residual set $\cR_f\subset M$,
if $x\in\cR_f$ then $\omega(x)$ is a Lyapunov stable set. Moreover, if a homoclinic class intersects $\omega(x)$, then they most
coincide 
\end{theorem}

\subsection{$C^1$ generic Dirac physical measures}

Finally, we state a $C^1$-generic consequence of Proposition~\ref{p.alrozero}

\begin{lemma}
\label{l.aldro}
There exists an open and dense set $\cA\subset\dif$ such that for every $f\in\cA$ and every $\sigma\in\fix(f)$, if $\delta_{\sigma}$ is a physical measure 
then $|\det(Df(\sigma))|<1$. 
\end{lemma}
\begin{proof}
Notice that the set $$\{f\in\dif;\textrm{for every hyperbolic fixed point}\:p,\:|\det(Df(p))|\neq 1\}$$ is open. Using Lemma~\ref{l.frank}
one easily gets that it is dense. The conclusion follows now from Proposition~\ref{p.alrozero}. 
\end{proof}

\section{Entropy estimation and Proof of Theorem~\ref{t.main}}\label{prova}

When we have a Lyapunov stable set with a dominated splitting, we can perform an entropy estimation using the result below. See \cite{CCE}.

\begin{theorem}[Catsigeras-Cerminara-Enrich]
\label{t.eleonora}
Let $\Lambda$ be a Lyapunov stable set for $f\in\dif$. Suppose that there exists a dominated splitting $T_{\Lambda}M=E\oplus F$.
Then, for every fisical measure $\mu$ supported in $\Lambda$ one has $$h_{\mu}(f)\geq\int\sum_{l=1}^{\dim F}\chi_l(x)d\mu.$$
\end{theorem}

\begin{proof}[Proof of Theorem \ref{t.main}]
Recall Kupka-Smale's theorem \cite{pdm} which asserts that there exists a residual $\cR_{KS}\subset\dif$ such that for every $f\in\cR_{KS}$, every 
periodic orbit is hyperbolic.
Pick $f\in\cR_{KS}\cap\cR_1\cap\cR_2\cap\cA$, let $\sigma\in\fix(f)$ be such that $\overline{\cB(\delta_{\sigma})}=M$ 
and $m(\cB(\delta_{\sigma}))>0$. 

Assume by contradiction that $\sigma$ is not a sink. Then it is a hyperbolic saddle since it obviously cannot be a source and $f\in\cR_{KS}$. 

Moreover, $f$ has no sinks, since otherwise $\cB(\delta_{\sigma})$ would intersect the basin of such a sink, which is impossible. 

By Lemma~\ref{l.aldro},
$|\det Df(\sigma)|<1$, and thus Lemma \ref{l.bobo} implies that the homoclinic class $H(\sigma)$ admits a 
dominated splitting $E\oplus F$ such that $\det(Df(\sigma)|_{F})>1$.  

Also, by Corollary~\ref{c.acumula} and Theorem~\ref{t.mp} we obtain a point $y\in M$ such that $\omega(y)$ is a Lyapunov stable set and $\sigma\in\omega(y)$. As a consequence $H(\sigma)\cap\omega(y)\neq\emptyset$.
Applying the second part of Theorem \ref{t.mp}, we conclude that the homoclinic class $H(\sigma)$ is a Lyapunov stable set.

Finally, since $\delta_{\sigma}$ is a physical measure supported in $H(\sigma)$, we can apply Theorem \ref{t.eleonora} and conclude that
$$h_{\delta_{\sigma}}(f)\geq\sum_{i=1}^{\dim F}\log\lambda_i,$$
where the numbers $\lambda_i$ are the modulus of the eigenvalues of $Df(\sigma)$ in the subspace $F$. Since $\det(Df(\sigma)|_{F})>1$, one obtains
that $h_{\delta_{\sigma}}(f)>0$, which is absurd. This completes the proof.
\end{proof}

\end{document}